\documentclass[10pt]{amsart}
\usepackage[T1]{fontenc}
\usepackage[utf8x]{inputenc}
\usepackage{dsfont}
\usepackage{amsmath,amssymb,amsfonts}
\usepackage{graphicx,soul,mathpazo,color,yfonts,paralist,stmaryrd}
\usepackage{amsxtra,latexsym,mathrsfs,verbatim}
\usepackage{epsfig,enumerate}
\usepackage{tikz}
\usepackage{hyperref}
\hypersetup{hidelinks,
  pdftitle={Nearly optimal packings of equally sized rainbow forests},
  pdfauthor={Boyan Xu}}

\newtheorem{theorem}{Theorem}
\newtheorem{lemma}[theorem]{Lemma}
\newtheorem{coro}[theorem]{Corollary}
\newtheorem*{mpsconj}{Conjecture}
\numberwithin{theorem}{section}
\numberwithin{equation}{section}
\newcommand{\E}{\mathbb E}
\newcommand{\Prob}{\mathbb P}
\newcommand{\cD}{\mathcal D}

\title[Equally sized rainbow forests]{Nearly optimal packings of equally sized rainbow forests}
\author{Boyan Xu}
\address[Xu]{School of Data Science and Information Engineering, Guizhou Minzu University,\newline
Guiyang, Guizhou Province, 550025, China}
\email{boyan04518@gmail.com}

\begin{document}
\begin{abstract}
For every $\varepsilon>0$, we prove that whenever $m$ is sufficiently
large in terms of $\varepsilon$ and $1\le k\le2m$, every simple graph
with $km$ edges and at most $m$ edges of each color contains at least
$(1-\varepsilon)m$ edge-disjoint forests, each with $k$ edges of distinct
colors. This gives the packing conjectured by Montgomery, Pokrovskiy,
and Sudakov throughout this range, even when adjacent edges may have the
same color. For every $k>2m$ there is a graph on $k$ vertices with $km$
edges of distinct colors and no $k$-edge forest. Applied to $K_n$, the
theorem gives $(1-o(1))\lfloor n/2\rfloor$ edge-disjoint rainbow spanning
trees whenever each color occurs on at most $\lfloor n/2\rfloor$ edges.
\end{abstract}
\maketitle

\section{Introduction}\label{sec:introduction}
How many edge-disjoint rainbow forests of size $k$ can one find in an
edge-colored graph with $km$ edges? Here a forest is \emph{rainbow} if
its edges have distinct colors, and its size is its number of edges.
There is room for at most $m$ such forests. If one can find at least
$(1-\varepsilon)m$ of them, they cover all but at most $\varepsilon km$ edges.
Throughout, graphs are finite and simple. An edge-coloring is
\emph{globally $m$-bounded} if each color occurs on at most $m$ edges,
and \emph{proper} if each color class is a matching.

Montgomery, Pokrovskiy, and Sudakov studied this question for proper
edge-colorings. They packed rainbow spanning trees in complete graphs
and, under additional edge-count and vertex-cover hypotheses, small
rainbow forests in general graphs~\cite[Theorem~1.14 and Lemma~9.5]{MPS}.
These results led them to the following conjecture.\footnote{The
published statement uses the phrase \emph{forests of order $k$}.
Benny Sudakov confirmed that $k$ counts edges.}

\begin{mpsconj}[Montgomery--Pokrovskiy--Sudakov~{\cite[Conjecture~10.2]{MPS}}]
For every $\varepsilon>0$, there is a positive integer $m$ such that, for
every positive integer $k$, every properly edge-colored, globally
$m$-bounded graph with $km$ edges contains at least $(1-\varepsilon)m$
pairwise edge-disjoint rainbow forests, each with $k$ edges.
\end{mpsconj}

The conjecture fails for every $k>2m$. In this range $km\le\binom{k}{2}$,
so we may take $km$ edges of $K_k$ and give each a different color.
Every forest in this graph has at most $k-1$ edges. At $k=2m+1$ the
example is $K_{2m+1}$ itself. For all $k$ up to and including $2m$, the
global color bound alone suffices.

\begin{theorem}\label{thm:main}
For every $\varepsilon>0$ there exists $m_0=m_0(\varepsilon)$ such that the
following holds for all integers $m\ge m_0$ and $1\le k\le2m$. Every
globally $m$-bounded edge-coloring of a graph with $km$ edges admits at
least $(1-\varepsilon)m$ pairwise edge-disjoint rainbow forests, each with
exactly $k$ edges.
\end{theorem}

The color bound is sharp even for matchings. Take a matching with
$m(m+1)$ edges and partition its edges into $m$ color classes of size $m+1$. Since there are
only $m$ colors, the matching has no rainbow forest of size $m+1$.

For spanning trees, Brualdi and Hollingsworth~\cite{BH} conjectured that
if the perfect matchings of a one-factorization of $K_{2m}$ receive distinct
colors, then the edges can be partitioned into $m$ rainbow spanning trees
whenever $m\ge3$. Glock, K\"uhn, Montgomery, and Osthus~\cite{GKMO} proved this for all
sufficiently large $m$, with all the trees isomorphic. Kaneko, Kano, and
Suzuki conjectured that every properly edge-colored $K_n$ with $n\ge6$
contains $\lfloor n/2\rfloor$ edge-disjoint rainbow spanning trees
(see~\cite[Conjecture~1.13]{MPS}). Pokrovskiy and Sudakov~\cite{PS} and,
independently, Balogh, Liu, and Montgomery~\cite{BLM} obtained a linear
lower bound. Montgomery, Pokrovskiy, and Sudakov~\cite[Theorem~1.14]{MPS}
then packed $(1-o(1))n/2$ such trees.

A proper edge-coloring of $K_n$ is automatically globally
$\lfloor n/2\rfloor$-bounded. Under this global bound alone, Akbari and
Alipour~\cite{AA} found two edge-disjoint rainbow spanning trees for
$n\ge5$, and Carraher, Hartke, and Horn~\cite{CHH} found
$\Omega(n/\log n)$. Theorem~\ref{thm:main} gives the asymptotic packing
bound under the same hypothesis.

\begin{coro}\label{cor:trees}
Every globally $\lfloor n/2\rfloor$-bounded edge-coloring of $K_n$ admits
at least $(1-o(1))\lfloor n/2\rfloor$ pairwise edge-disjoint rainbow
spanning trees as $n\to\infty$.
\end{coro}
\begin{proof}
Put $m=\lfloor n/2\rfloor$ and $k=n-1$. Since $km\le\binom n2$ and
$k\le2m$, we may retain any $km$ edges and apply Theorem~\ref{thm:main}.
A forest with $n-1$ edges on $n$ vertices is a spanning tree.
\end{proof}

Corollary~\ref{cor:trees} gives one instance of the extension to globally
bounded colorings asked for in~\cite[Section~10]{MPS}. The rainbow Hamilton
cycle packings of Kim, K\"uhn, Kupavskii, and Osthus~\cite{KKKO} in
quasirandom graphs assume both local and global color bounds.

Our proof reduces to $k=2m$ by adding monochromatic matchings in new
colors. One rainbow forest of this size is easy to find. Give every edge
weight $1/m$. Each color class then has weight at most one.
A set of $u$ vertices, $1\le u\le2m$, spans weight at most
$\binom u2/m\le u-1$. For $u>2m$ the total weight $2m$ suffices.
Edmonds' matroid intersection theorem~\cite{Edmonds}, applied to the
partition matroid of the color classes and the graphic matroid truncated at
$2m$, expresses these weights as a convex combination of incidence vectors
of rainbow forests of size at most $2m$. Their mean size is $2m$, so each
forest in the combination has exactly $2m$ edges.

Even these uniform edge probabilities can lead to a poor first choice.
Add a disjoint matching of size $m$, all in one new color, to a properly
edge-colored $K_{2m}$. Take the star at a uniformly chosen clique vertex
and one uniformly chosen matching edge. Each edge belongs to the resulting
rainbow forest with probability $1/m$. After its deletion all remaining
clique edges lie on $2m-1$ vertices, and a rainbow forest can use at most
one matching edge. A second rainbow forest of size $2m$ is impossible.

To make successive choices, we keep track of how many edges remain in
each member of a family given by Lemma~\ref{lem:tests}. These counts
control the number of edges inside a small vertex set and the number with
an end outside a set of nearly $2m$ vertices. They let us sample weights
that satisfy the color and forest constraints and again apply matroid
intersection. An exponential potential lets us choose a forest whose
deletion preserves the estimates for the next step.

\section{Edge counts in induced subgraphs}\label{sec:induced}
For a graph $H$ and $U\subseteq V(H)$, let $E(U)$ be the set of edges with both
ends in $U$. Put $F_U=E(H)\setminus E(U)$ and, for a positive integer $m$,
$q_U=|F_U|/m$. If $\nu$
is a real-valued function on $E(H)$, write $\nu(D)=\sum_{e\in D}\nu_e$
for $D\subseteq E(H)$.

\begin{lemma}\label{lem:tests}
Let $m$ be a sufficiently large integer, and let $H$ be a graph on $N$ vertices
with $2m^2$ edges and no isolated vertices. There is a family $\cD$ of
edge sets such that $|D|\leq2m$ for each $D\in\cD$ and
$\log(2+|\cD|)=O(\sqrt m\log m)$. For every $\eta\geq0$ and every real
edge weight $\nu$ satisfying $|\nu(D)|\leq\eta$ for all $D\in\cD$, we have
\[
 |\nu(E(U))|\le 4\eta |U|\quad (|U|\le 4m),\qquad
 |\nu(F_U)|\le 80\eta q_U\quad (m<|U|\le 2m).
\]
For $Q\ge1$, at most $(N+1)^{\lceil4Q\rceil}$ sets $U$ in the second range
satisfy $q_U\le Q$.
\end{lemma}

\begin{proof}
We have $2m<N\le4m^2$. Choose a set $C$ of $2m$ vertices spanning as
many edges as possible, and put $w=\lceil\sqrt{2m}\rceil$. Take for $\cD$
all sets $E(W)$ with $|W|\le w$, the set $D_v$ of edges joining $v$ to a
vertex of $C$ for each $v\in V(H)$, and the parts of a partition of $F_C$
into $\lceil q_C\rceil$ sets of at most $m$ edges. Each member of $\cD$
has size at most $2m$, and $|\cD|\le (N+1)^w+N+\lceil q_C\rceil$.
If $w<u=|U|\le4m$, a uniformly chosen $w$-subset of $U$ contains both
ends of each edge in $E(U)$ with probability $\binom{w}{2}/\binom{u}{2}$. Averaging gives
$|\nu(E(U))|\le\eta\binom{u}{2}/\binom{w}{2}\le4\eta u$.
For $1\le u\le w$ we have $E(U)\in\cD$, and the bound for $u=0$ is immediate.

Now let $m<u\le2m$. Extending $U$ to a $2m$-set gives $q_C\le q_U$,
and $|E(C)|\le\binom{2m}{2}=2m^2-m$ gives $q_C\ge1$.
Since $E(C)\cap E(U)=E(C\cap U)$, inclusion--exclusion and simplicity give
\[
 2m^2-m(q_C+q_U)=|E(C)|+|E(U)|-|E(H)|
 \le |E(C\cap U)|\le \binom{|C\cap U|}{2}\le m|C\cap U|.
\]
Put $R=C\setminus U$, $S=U\setminus C$ and $T=R\cup S$. Then
$|R|=2m-|C\cap U|\le q_C+q_U\le2q_U$. Since $|U|\le|C|$, we have
$|S|\le|R|$ and hence $|T|\le 2|R|\le4q_U$.
The set $T$ determines $U$, since $U=(C\setminus T)\cup(T\setminus C)$.
Consequently at most $(N+1)^{\lceil4Q\rceil}$ such sets $U$ have $q_U\le Q$.

\begin{figure}[htb]
\centering
\begin{tikzpicture}[x=1cm,y=1cm,font=\normalsize,line width=.45pt]
  \fill[black!9] (0,0) rectangle (2.4,.82);
  \fill[black!9] (4.8,0) rectangle (7.2,.82);
  \draw (0,0) rectangle (7.2,.82);
  \draw (2.4,0) -- (2.4,.82) (4.8,0) -- (4.8,.82);
  \node at (1.2,.28) {$R=C\setminus U$};
  \node at (3.6,.28) {$C\cap U$};
  \node at (6.0,.28) {$S=U\setminus C$};
  \draw[dashed] (1.2,.61) -- (6.0,.61);
  \fill (1.2,.61) circle (1.3pt) (6.0,.61) circle (1.3pt);
  \draw (0,.90) -- (0,1.03) -- (4.8,1.03) -- (4.8,.90);
  \node[above,inner sep=1.3pt] at (2.4,1.03) {$C$};
  \draw (2.4,-.08) -- (2.4,-.21) -- (7.2,-.21) -- (7.2,-.08);
  \node[below,inner sep=1.3pt] at (4.8,-.21) {$U$};
\end{tikzpicture}
\caption{The shaded parts form $T=R\cup S$. The weight of the dashed edge
appears in $\sum_{v\in S}\nu(D_v)$, although neither $C$ nor $U$ spans it.}
\label{fig:sets}
\end{figure}

Write $g(A)=\nu(E(A))$. Removing $R$ from $C$ subtracts
$\sum_{v\in R}\nu(D_v)-g(R)$, since the sum counts each edge inside $R$
twice. To add $S$ to $C\cap U$, we need the edges inside $S$ and those
joining $S$ to $C\cap U$. The sum $\sum_{v\in S}\nu(D_v)$ includes the
latter, together with all edges joining $S$ to $R$ (Figure~\ref{fig:sets}).
Those extra edges have total weight $g(T)-g(R)-g(S)$. It follows that
\[
 g(U)-g(C)=\sum_{v\in S}\nu(D_v)-\sum_{v\in R}\nu(D_v)+2g(R)+2g(S)-g(T).
\]
The two sums contribute at most $\eta|T|$ in absolute value. Since
$|R|+|S|=|T|\le4m$, the first bound gives
$2|g(R)|+2|g(S)|+|g(T)|\le12\eta|T|$. The partition of $F_C$ gives
$|\nu(F_C)|\le\lceil q_C\rceil\eta\le2q_U\eta$. Since
$\nu(F_U)=\nu(F_C)+g(C)-g(U)$, we obtain
$|\nu(F_U)|\le (2+13\cdot4)\eta q_U\le80\eta q_U$.
\end{proof}

\section{Packing the forests}\label{sec:packing}
\begin{proof}[Proof of Theorem~\ref{thm:main}]
We may assume $0<\varepsilon<1/2$. It suffices to prove the theorem for
$k=2m$. Indeed, adjoin $2m-k$ disjoint matchings, each with $m$ edges, on
new vertices. Assign a distinct new color to each matching. The enlarged
graph has $2m^2$ edges. Each rainbow $2m$-edge forest uses at most one
edge of each new color, so it contains at least $k$ original edges.
Keeping any $k$ of them in each forest proves the reduction. Henceforth
let $G$ have $2m^2$ edges.

\smallskip\noindent\emph{Reduction at the endpoint.}
Put $\alpha=\varepsilon/4$ and $d=\lceil\alpha m\rceil$. For this reduction
write $q_U(J)=|E(J)\setminus E(J[U])|/m$ when the graph is $J$. By moving
at most $d$ edges of $G$ to new vertices, we shall obtain a graph $H$ with
\begin{equation}\label{eq:slack}
 q_U(H)\ge(1+\alpha)(2m-|U|+1)\qquad(m<|U|\le2m).
\end{equation}
For $m<u=|U|<2m$ this holds in every graph $J$ with $2m^2$ edges. Indeed,
putting $j=2m-u+1$ and using simplicity, we have
\[
 q_U(J)-j=u-1-|E(J[U])|/m
 \ge (u-1)(2m-u)/(2m)\ge(j-1)/2\ge j/4.
\]
It remains to consider the sets of size $2m$. If each such set $U$ has
$q_U(G)\ge1+\alpha$, take $H=G$. Otherwise choose $C$ of size $2m$ with
$q_C(G)<1+\alpha$. It spans more than $2m^2-(1+\alpha)m$ edges. Delete
$d$ of these edges and replace them with $d$ independent edges on $2d$
new vertices, giving each replacement the color of the edge it replaces.
Call the resulting graph $H$.

In the latter case every $2m$-set $U$ in $H$ contains at least $d$
nonadjacent pairs. For $U=C$ take the pairs formed by the deleted edges.
Otherwise pick $v\in U\setminus C$. Each old vertex outside $C$ has
degree at most $mq_C(G)<(1+\alpha)m$, and each new vertex has degree
one. Thus $v$ has at least $2m-1-(1+\alpha)m\ge d$ nonneighbors in $U$
for large $m$.
For $|U|=2m$, the number $mq_U(H)$ equals $m$ plus the number of
nonadjacent pairs in $U$. This proves \eqref{eq:slack}. We now work in
$H$ and write $q_U=q_U(H)$. Our aim is to choose
$r=\lceil(1-\varepsilon/2)m\rceil$ edge-disjoint rainbow forests of size $2m$.

\smallskip\noindent\emph{Choosing one forest.}
Delete isolated vertices of $H$, put $N=|V(H)|$, and fix the family
$\cD$ of Lemma~\ref{lem:tests}.
After $t<r$ forests have been removed, let $E_t$ be the remaining edge set
and $s=m-t$. We have $s\ge\varepsilon m/2$. We maintain $|E_t|=2ms$,
at most $s$ edges of each color, and, with $\eta=\alpha/1000$,
\begin{equation}\label{eq:track}
 |x_D(t)|\le\eta\quad(D\in\cD),\qquad
 x_D(t)=\frac{|D\cap E_t|}{s}-\frac{|D|}{m}.
\end{equation}
These conditions hold at $t=0$. All implicit constants below may depend on
$\varepsilon$.

For this step, whenever two nonempty color classes have total size at most
$s$, merge that pair. When this process stops, each resulting class has
size at most $s$, and at most one has size at most $s/2$.
Put $L=m^{1/8}$ and $b=\lfloor m/L\rfloor$. Partition each class
of size $\ell$ into $\lceil\ell/b\rceil$ blocks whose sizes differ by at
most one. From each block $B$ choose one edge independently and uniformly,
and give it weight $|B|/s$. Give every unchosen edge of $H$ weight zero.
Write $X$ for these weights and $J$ for their support. Each remaining edge has expected
weight $1/s$. The total weight is $2m$, and each merged color class has weight equal
to its size divided by $s$.

Let $\mathcal A$ be the event that $X(E(U))\le |U|-1$ for every nonempty
$U\subseteq V(H)$. We show that $\Prob(\mathcal A^c)\le m^{-100}$.
For a fixed edge set $A$, a block $B$ contributes $|B|/s$ to $X(A)$ if its
chosen edge belongs to $A$, and zero otherwise. These contributions are
independent and bounded by $h=b/s=O(1/L)$. Thus, writing $\mu=\E X(A)$,
the Chernoff bound applied to $X(A)/h$ gives, for $a>0$,
\[
 \Prob\bigl(|X(A)-\mu|\ge a\bigr)
 \le 2\exp\!\left(-\frac{a^2}{h(2\mu+a)}\right).
\]

Put $u=|U|$ and $\nu_e=\mathds{1}_{\{e\in E_t\}}/s-1/m$. Then
$\E X(A)=|A|/m+\nu(A)$ for every edge set $A$, and \eqref{eq:track}
gives $|\nu(D)|\le\eta$ for every $D\in\cD$. For $2\le u\le m$,
Lemma~\ref{lem:tests} and simplicity give
\[
 \E X(E(U))\le \binom{u}{2}/m+4\eta u\le\tfrac34(u-1).
\]
The threshold $u-1$ exceeds the mean by at least $(u-1)/4$, so the
probability of a violation is at most $2e^{-c_\varepsilon Lu}$ for some
$c_\varepsilon>0$.

For $m<u\le2m$, use $X(E(U))=2m-X(F_U)$. In this case a violation means
that $X(F_U)<2m-u+1$. By Lemma~\ref{lem:tests}, its mean $\mu=\E X(F_U)$
satisfies $|\mu-q_U|\le80\eta q_U$, and \eqref{eq:slack} gives
\[
 \mu-(2m-u+1)\ge
 \bigl(1-80\eta-(1+\alpha)^{-1}\bigr)q_U\ge \alpha q_U/2.
\]
Since $\mu\le 2q_U$, the same Chernoff bound gives a probability at most
$2e^{-c_\varepsilon Lq_U}$, decreasing $c_\varepsilon$ if necessary.

There are at most $N^u$ sets of size $u$. For each integer $i\ge1$,
Lemma~\ref{lem:tests} gives at most $(N+1)^{4(i+1)}$ sets in the second
range with $i\le q_U<i+1$. The union bound over the two ranges is at most
\[
 2\sum_{u=2}^m N^u e^{-c_\varepsilon Lu}
 +2\sum_{i\ge1}(N+1)^{4(i+1)}e^{-c_\varepsilon Li}.
\]
As $N\le4m^2$ and $L=m^{1/8}$, for large $m$ each counting factor is at
most $e^{c_\varepsilon Lu/2}$ or $e^{c_\varepsilon Li/2}$, respectively.
The two geometric sums are then less than $m^{-100}$ in total. Singletons
span no edges, and for $u>2m$ the total weight $2m$ is at most $u-1$.
This proves the bound on $\Prob(\mathcal A^c)$.

On $\mathcal A$, $X$ satisfies the rank inequalities of two matroids on
$E_t$. For the partition matroid of the merged color classes this follows
from the weight bound of one per class. For the graphic matroid, take
$A\subseteq E_t$ and sum $X(E(W))\le |W|-1$ over the vertex sets $W$ of
the components of $(V(H),A)$. Every edge of $A$ is counted, and the
upper bound $\sum_W(|W|-1)$ is the rank of $A$. Since also
$X(A)\le X(E_t)=2m$, the rank inequalities hold for the graphic matroid
truncated at $2m$. By Edmonds' matroid intersection
theorem~\cite[Corollary~41.12b]{Schrijver}, $X$ is a convex combination
of incidence vectors of common independent sets.

Let $X^*$ have the law of $X$ conditioned on $\mathcal A$. For each value
of $X^*$, choose $P$ using the coefficients of one such convex combination
as probabilities. Then $\Prob(e\in P\mid X^*)=X^*_e$, so $P$ uses only
edges of positive weight. As $|P|\le2m$ and $\E(|P|\mid X^*)=2m$, we have
$|P|=2m$ almost surely. Each merged color class contributes at most one edge,
so $P$ is rainbow in the original coloring. An original class of size $s$
remained separate and has weight one, so $P$ meets it almost surely.
Deleting $P$ thus leaves $2m(s-1)$ edges and at most $s-1$ of each color.

For each $e\in E_t$, we have $\Prob(e\in P)=\E X^*_e=\E(X_e\mid\mathcal A)$.
Since $0\le X_e\le1$, conditioning on $\mathcal A$ changes this mean by
at most $\Prob(\mathcal A^c)$. Recalling that $\E X_e=1/s$, we obtain
\begin{equation}\label{eq:marg}
 \bigl|\Prob(e\in P)-1/s\bigr|\le\Prob(\mathcal A^c)\le m^{-100}.
\end{equation}

\smallskip\noindent\emph{Iterating the choice.}
Fix $D\in\cD$ and write $K=|P\cap D|$. The inclusion probabilities
in \eqref{eq:marg} control the mean of $K$. To control its fluctuations
we return to the blocks. All but at most one block have size at least
$b/3$. Indeed, a class of size $\ell\ge b$ has blocks of size at least
$\lfloor\ell/\lceil\ell/b\rceil\rfloor\ge b/3$ for large $m$, and at most
one class has size less than $b<s/2$. In the original block experiment
$W=|J\cap D|$ is a sum of independent indicators with
$\E W\le1+3|D|/b=O(L)$. Hence
$\E e^{W/L}\le\exp((e^{1/L}-1)\E W)=O(1)$. Since $P$ uses only edges
of positive weight in $X^*$, we obtain
\[
 \E e^{K/L}\le\E(e^{W/L}\mid\mathcal A)
 \le\frac{\E e^{W/L}}{\Prob(\mathcal A)}=O(1).
\]

Put $\lambda=m^{5/8}$ and let
\[
 \Phi_t=\sum_{D\in\cD}\bigl(e^{\lambda x_D(t)}+e^{-\lambda x_D(t)}\bigr).
\]
For our fixed $D$, put $z=|D\cap E_t|/s\le2m/s=O(1)$ and
$y=x_D(t+1)-x_D(t)=(z-K)/(s-1)$. Summing
\eqref{eq:marg} over $D\cap E_t$ gives $\E K=z+O(m^{-99})$ and hence
$|\E y|=O(m^{-100})$. Put $v=(K+z)/L$. We have $\E e^v=O(1)$ and
$|y|\le Lv/(s-1)$. As $\lambda L/(s-1)=O(m^{-1/4})$, we have
$\lambda|y|\le v/2$ for large $m$. Using $v^2e^{v/2}=O(e^v)$, we get
$\E[y^2 e^{\lambda|y|}]=O(L^2/m^2)$. Since $\lambda^2 L^2/m^2=m^{-1/2}$,
Taylor's inequality gives
\[
 \E e^{\pm\lambda y}
 \le 1+\lambda|\E y|+\tfrac12\lambda^2\E[y^2 e^{\lambda|y|}]
 \le 1+O(m^{-1/2}).
\]
For any fixed history satisfying \eqref{eq:track} and the edge and color
counts, multiply by $e^{\pm\lambda x_D(t)}$ and sum over $D$ and both
signs. This gives $\E\Phi_{t+1}\le(1+C_\varepsilon/\sqrt m)\Phi_t$,
with $C_\varepsilon$ independent of the history.

We can therefore choose the forests so that, for $0\le t\le r$,
\[
 \log\Phi_t\le\log(2|\cD|)+C_\varepsilon t/\sqrt m .
\]
At $t=0$ this is equality. Suppose the bound holds at some $t<r$. Since
$t\le m$, Lemma~\ref{lem:tests} gives $\log\Phi_t=O(\sqrt m\log m)$.
For every $D\in\cD$ we have $|x_D(t)|\le\lambda^{-1}\log\Phi_t$,
so \eqref{eq:track} holds for large $m$. Choose a forest for which
$\Phi_{t+1}$ is at most its expected value. The edge and color counts
hold after its deletion, and $\log(1+C_\varepsilon/\sqrt m)\le
C_\varepsilon/\sqrt m$ gives the bound at $t+1$.

At most $d$ of the $r$ edge-disjoint forests in $H$ use a replacement
edge. The others lie in $G$, and their number is at least
$r-d\ge(1-3\varepsilon/4)m-1\ge(1-\varepsilon)m$ for large $m$.
Keeping $k$ original edges in each forest, as in the initial reduction,
proves the theorem.
\end{proof}

\smallskip\noindent\textbf{Acknowledgment.}
The author thanks Benny Sudakov for clarifying the statement of Conjecture~10.2 in~\cite{MPS}.

\smallskip
\noindent\textbf{Use of generative AI.}
The author used ChatGPT and Codex during the preparation of this paper. Discussions with ChatGPT helped shape the argument in Section~2 and improve the presentation of the paper. Codex played an important role in identifying Lemma~\ref{lem:tests} and developing the shorter proof built around it, which led to the extension from proper edge-colorings to arbitrary edge-colorings satisfying the same global color bound. The author carefully reviewed and verified the mathematical arguments, checked the references, and made all final decisions concerning the paper. The author takes full responsibility for its correctness and content.

\bibliographystyle{plain}
\bibliography{references}
\end{document}